\documentclass[11pt,reqno]{amsart}
\usepackage{hyperref}
\usepackage{enumerate,fullpage}
\usepackage{amssymb,amsmath,graphicx,amsthm}
\usepackage{color}

\newtheorem{theorem}{Theorem}
\newtheorem*{theorem1}{Theorem}
\newtheorem{corollary}[theorem]{Corollary}
\newtheorem{lemma}[theorem]{Lemma}
\newtheorem{proposition}[theorem]{Proposition}
\theoremstyle{remark}
\newtheorem{remark}{Remark}
\theoremstyle{definition}
\newtheorem{definition}{Definition}

\newcommand{\Om}{\Omega}

\newcommand{\p}{\partial}

\newcommand{\ep}{\epsilon}

\def\di{\partial}
\def\dib{\bar\partial}

\def\T{\text}
\def\C{\mathbb C}

\begin{document}

	\title[Iterates of holomorphic self-maps]{Iterates of holomorphic self-maps on  pseudoconvex domains of finite and infinite type in $\mathbb C^n$}
	
	\author[T.V. Khanh and N.V. Thu]{ Tran Vu Khanh and Ninh Van Thu} 
	
	\address{Tran Vu Khanh}
	\address{School of Mathematics and Applied Statistics, University of Wollongong, NSW, Australia,  2522}
	\email{tkhanh@uow.edu.au}
	\address{Ninh Van Thu}
	\address{Department of Mathematics, Vietnam National University at Hanoi, 334 Nguyen Trai, Thanh Xuan, Hanoi, Vietnam}
	\email{thunv@vnu.edu.vn}

	\begin{abstract}  Using the lower bounds on the Kobayashi metric established by the first author \cite{Kha14}, we prove the Wolff-Denjoy-type theorem for a very large class of pseudoconvex domains in $\C^n$ that may contain many classes of pseudoconvex domains of finite type and infinite type.
	\end{abstract}
	\maketitle
	
	\section{Introduction}

	In 1926, Wolff \cite{Wolff} and Denjoy \cite{Denjoy} established their famous theorem regarding the behavior of iterates of holomorphic self-mapings without fixed points of the unit disk $\Delta$ in the complex plan.  
	\begin{theorem1}[Wolff-Denjoy \cite{Wolff, Denjoy}, 1926] Let $\phi:\Delta\to \Delta$ be a holomorphic self-map without fixed points. Then there exists a point $x$ in the unit circle $\partial\Delta$ such that the sequence $\{\phi^k\}$ of iterates of $\phi$ converges, uniformly on any compact subsets of $\Delta$, to the constant map taking the value $x$.
	\end{theorem1}
	The generalization of this theorem to domains in $\C^n$, $n\ge 2$, is clearly a natural problem. This has been done in several cases:
\begin{itemize}
	\item the unit ball (see \cite{Herv});
	\item strongly convex domains (see \cite{A1, A4, A-R});
\item strongly pseudoconvex domains (see \cite{A3, Hua});
\item pseudoconvex domains of strictly finite type in the sense of Range \cite{Ra} (see \cite{A3}) ;
\item pseudoconvex domains of finite type in $\C^2$ (see \cite{Kar, Zha}).
\end{itemize}
 The main goals of this paper is to prove the Wolff-Denjoy-type theorem for a very general class of bounded pseudoconvex domains in $\C^n$ that may contain  many classes of pseudoconvex domains of finite type and also infinite type. In particular, 	we shall prove that (the definitions are given below)
 \begin{theorem}\label{main} Let $\Omega\subset\mathbb C^n$ be a bounded, pseudoconvex domain with $C^2$-smooth boundary $\partial\Omega$. Assume that 
 	\begin{enumerate}
 		\item[(i)] $\Om$ has the $f$-property with $f$ satisfying $\displaystyle\int_1^\infty\frac{\ln \alpha}{\alpha f(\alpha)}d\alpha<\infty$ ; and
 		\item[(ii)] the Kobayashi distance of $\Om$ is complete.
 	\end{enumerate}Then, if $\phi:\Om\to\Om$ is a holomorphic self-map such that the sequence of iterates $\{\phi^k\}$ is compactly divergent, then the sequence $\{\phi^k\}$ converges, uniformly on a compact set, to a point of the boundary.
 \end{theorem}
We say that the Wolff-Denjoy-type theorem for $\Om$ holds if the conclusion of Theorem~\ref{main} holds. Following the work by Abate \cite{A1,A3, A4} and using the estimate of the Kobayashi distance on domains of the $f$-property, we will prove Theorem~\ref{main} in Section 3. \\

 Here we have some remarks on the $f$-property and on the completeness of the Kobayashi distance. The $f$-property is defined in \cite{Khanh, Kha14} as 
 	\begin{definition} 
 		\label{d1}  We say that domain $\Om$ has the $f$-property if there exists a family of functions $\{\psi_\eta\}$ such that
 		\begin{enumerate}
 			\item [(i)] $|\psi_\eta| \le1$, $C^2$, and plurisubharmonic on $\Om$;
 			\item[(ii)] $i\di\dib \psi_\eta \ge c_1f(\eta^{-1})^2\T{Id}$ and $|D\psi_\delta|\le c_2\eta^{-1}$ on   $\{z\in \Om:-\eta<\delta_\Om(z)<0\}$ for some constants $c_1,c_2>0$,  where $\delta_\Om(z)$ is the euclidean distance from $z$ to the boundary $\di\Om$.
 		\end{enumerate} 
 		 	\end{definition}
  This is an analytic condition where the function $f$ reflects the geometric ``type" of the boundary. For example, by Catlin's results on pseudoconvex domains of finite type through the lens of the $f$-property \cite{Cat83, Cat87}, $\Om$ is of finite if and only if there exists an $\epsilon>0$ such that the $t^\epsilon$-property holds. If domain is reduced to be convex of finite type $m$, then the $t^{1/m}$-property holds \cite{McNeal}. Furthermore, there is a large class of infinite type pseudoconvex domains that satisfy an  $f$-properties \cite{ Khanh, Kha14}. For example (see \cite{Khanh}), 
	the $\log^{1/\alpha}$-property holds for both the complex ellipsoid of infinite type
	\begin{eqnarray}\label{complexellipsoid}
	\Om=\left\{z\in\C^n: \sum_{j=1}^n \exp\left(-\frac{1}{|z_j|^{\alpha_j}}\right)-e^{-1}<0\right\}
	\end{eqnarray}
with $\alpha:=\max_{j}\{\alpha_j\}$, and the real ellipsoid of infinite type
	\begin{eqnarray}\label{realellipsoid}
	\tilde\Om=\left\{z=(x_1+iy_{1},\dots,x_n+iy_{n})\in\C^n: \sum_{j=1}^{n} \exp\left(-\frac{1}{|x_j|^{\alpha_j}}\right)+\exp\left(-\frac{1}{|y_j|^{\beta_j}}\right)-e^{-1}<0\right\}
 \end{eqnarray}
with $\alpha:=\max_{j}\{\min\{\alpha_j,\beta_j\}\}$, where $\alpha_j,\beta_j>0$ for all $j=1,2,\dots$. The influence of the $f$-property on estimates of the Kobayashi metric and distance will be given in Section 2. \\

	The completeness of the Kobayashi distance (or $k$-completeness for short) is a natural condition of hyperbolic manifolds. The qualitative condition for the $k$-completeness of a bounded domain $\Om$ in $\C^n$ is the Kobayashi distance 
$$k_\Om(z_0,z)\to \infty\quad \T{as}\quad z\to \di\Om$$
 for any point $z_0\in \Om$. By literature, it is well-known that this condition holds for strongly pseudoconvex domains \cite{F-R}, or convex domains \cite{Mer}, or pseudoconvex domains of finite type in $\C^2$ \cite{Zha}, pseudoconvex Reinhardt domains \cite{Wa}, or domains enjoying a local holomorphic peak  function at any boundary point \cite{G}. We also remark that the domain defined by \eqref{complexellipsoid} (resp. \eqref{realellipsoid}) is $k$-complete because it is a pseudoconvex Reinhardt domain (resp. convex domain). These remarks immediately lead to the following corollary.
\begin{corollary}\label{cor1} Let $\Om$ be a bounded domain in $\C^n$. The Wolff-Denjoy-type theorem  for $\Om$ holds if $\Om$ satisfies at least one of the following settings:
	\begin{enumerate}
		\item[(a)] $\Om$  is a strongly pseudoconvex domain;
  \item[(b)]  $\Om$ is a pseudoconvex domains of finite type and $n=2$;
  \item[(c)] $\Om$ is a convex domain  of finite type;
  \item[(d)] $\Om$ is a pseudoconvex Reinhardt domains of finite type;
  \item[(e)] $\Om$ is a pseudoconvex domain of finite type (or of infinite type having the $f$-property with $f(t)\ge \ln^{2+\epsilon}(t)$ for any $\epsilon>0$) such that $\Om$ has a local, continuous, holomorphic peak function at each boundary point, i.e., for any $x\in \di\Om$ there exist a neighborhood $U$ of $p$ and a holomorphic function $p$ on $\Om\cap U$, continuous up to $\bar\Om\cap U$, and satisfies
  $$p(x)=1,\quad p(z)<1,\quad\T{for all~~} z\in \bar\Om\cap U\setminus\{x\};$$
  	\item[(f)] $\Om$ is defined by \eqref{complexellipsoid} or \eqref{realellipsoid} with $\alpha<\frac{1}{2}$.
\end{enumerate}
\end{corollary}
	Finally, throughout the paper we use  $\lesssim$ and $\gtrsim$ to denote inequalities up to a positive constant, and $H(\Om_1,\Om_2)$ to denote the set of holomorphic maps from $\Om_1$ to $\Om_2$.
	
	\section{The Kobayashi metric and distance}
	
We start this section by the definition of Kobayashi metric.

	\begin{definition}
	Let $\Om$ be a domain in $\C^n$, and $T^{1,0}\Om$ be its holomorphic tangent bundle.
 The Kobayahsi (pseudo)metric $K_\Om: T^{1,0}\Omega\to \mathbb R$ is defined by 
	\begin{eqnarray}
	\begin{split}
	K_\Om(z,X)=&\inf\{ \alpha>0~|~\exists~ \Psi\in H(\Delta,\Om): \Psi(0)=0, \Psi'(0)=\alpha^{-1}X\},
	\end{split}
	\end{eqnarray}
for any $z\in \Om$ and $X\in T^{1,0}\Om$, where $\Delta$ be the unit open disk of $\C$.  
\end{definition}
	For the case that $\Omega$ is a smoothly pseudoconvex bounded domain of finite  type,  it is known that there exist $\ep>0$ such that the Kobayashi metric $K_\Omega$ 
	has the lower bound  $\delta^{-\epsilon}(z)$ (see \cite{Cho}, \cite{D-F}), in other word, 
	$$ K_\Om(z,X)\gtrsim \frac{|X|}{\delta^\epsilon_{\Om}(z)},$$
where $|X|$ is the euclidean length of $X$. 
	 Recently, the first author \cite{Kha14} obtained lower bounds on the Kobayashi metric for a general class of pseudoconvex domains in $\mathbb C^n$, that contains all domains of finite type and many domains of infinite type. 
	
	\begin{theorem}\label{k1}Let $\Om$ be a pseudoconvex domain in $\C^n$ with $C^2$-smooth boundary $\di\Om$. Assume that $\Om$ has the $f$-property with $f$ satisfying $\displaystyle\int_t^\infty \dfrac{d\alpha}{\alpha f(\alpha)}<\infty$ for some $t\ge 1$, and denote by $(g(t))^{-1}$ this finite integral. Then, 
		\begin{eqnarray}\label{K1}
		K(z,X) \gtrsim g(\delta^{-1}_{\Om}(z))|X|
		\end{eqnarray}
		for any $z\in\Om$ and $X\in T^{1,0}_z\Om$.
	\end{theorem}
	
	The Kobayashi (pseudo)distance $k_\Om:\Om\times\Om\to \mathbb R^+$ on $\Om$ is the integrated form of $K_\Om$, and given by 
	$$k_\Om(z,w)=\inf\left\{\int_a^bK_\Om(\gamma(t),\dot\gamma(t))dt~\Big|~ \gamma:[a,b]\to\Om, \T{piecwise $C^1$-smooth curve}, \gamma(a)=z,\gamma(b)=w \right\}$$
for any $z,w\in \Om$. The particular property of $k_\Om$ that it is contracted by holomorphic maps, i.e.,
\begin{eqnarray}\label{contracted}
\T{if} \quad\phi\in H(\Om,\tilde\Om)\quad \T{then}\quad k_{\tilde\Om}(\phi(z),\phi(w))\le k_\Om(z,w),\quad\T{for all}\quad  z,w\in\Om. 
\end{eqnarray}
	
We need the following lemma in \cite{A0,F-R}.
\begin{lemma}\label{l0}
	Let $\Om$ be a bounded $C^2$-smooth domain in $\C^n$ and $z_0\in \Om$. Then 
	there exists a constant $c_0>0$ depending on $\Om$ and $z_0$ such that
	$$k_\Om(z_0,z)\le c_0-\frac{1}{2}\log \delta(z,\di\Om)$$
	 for any $z\in \Om$.
\end{lemma} We recall that the curve $\gamma:[a,b]\to\Om$ is called a minimizing geodesic with respect to Kobayashi metric between two point $z=\gamma(a)$ and $w=\gamma(b)$ if 
$$k_\Om(\gamma(s),\gamma(t))=t-s=\int_{s}^t K_{\Om}(\gamma(t),\dot\gamma(t))dt, \quad \T{for any } s,t\in [a,b], s\le t.$$
This implies that 
$$K(\gamma(t),\dot\gamma(t))=1,\quad \T{for any } t\in [a,b].$$
	The relation between the Kobayashi distance $k_\Om(z,w)$ and the euclidean distance $\delta_\Om(z,w)$ will be expressed by the following lemma, which is a generalization of \cite[Lemma $36$]{Kar}.
		\begin{lemma}\label{Lem1} Let $\Om$ be a bounded, pseudoconvex, $C^2$-smooth domain in $\C^n$ satisfying the $f$-property with $\displaystyle\int_1^\infty\frac{\ln \alpha}{\alpha f(\alpha)}d\alpha<\infty$ and $z_0\in\Om$.   Then, there exists a constant $c$ only depending on $z_0$ and $\Om$ such that
		\begin{equation}\label{eq1207}
		\delta_\Om(z, w)\le c\int_{e^{2k_{\Om}(z_0,\gamma)-2c_0}}^\infty\frac{\ln \alpha}{\alpha f(\alpha)}d\alpha.		\end{equation}
		for all $z,w\in \Omega$, where $\gamma$ is a minimizing geodesic connecting $z$ to $w$ and $c_0$ is the constant given in Lemma~\ref{l0}.
	\end{lemma}
	\begin{proof} We only need to consider $z\not= w$ otherwise it is trivial. 
		Let $p$ be a point on $\gamma$ of minimal distance to $z_0$. We can assume that $p\not=z$ (if not, we interchange $z$ and $w$) and denote by $\gamma_1:[0,a]\to \Om$ the parametrized piece of $\gamma$ going from $p$ to $z$. By the minimality of $k_\Om(z_0,\gamma)=k_{\Om}(z_0,p)$ and the triangle inequality we have 
		\begin{eqnarray}\label{triangle}
		k_\Om(z_0,\gamma_1(t))\ge k_\Om(z_0,\gamma) \quad \T{and }\quad k_\Om(z_0,\gamma_1(t))\ge k_\Om(p,\gamma(t))-k_\Om(z_0,p)=t-k_\Om(z_0,\gamma)
		\end{eqnarray}  
		for any $t\in[0,a]$.
		Substituting $z=\gamma_1(t)$ into the inequality in Lemma~\ref{l0}, it follows  
		$$
		\frac{1}{\delta_\Om(\gamma_1(t))}\geq e^{2k_\Om(z_0,\gamma_1(t))-2c_0}
		$$
		for all $t\in [0,a]$. 	 Since $\gamma_1$ is a unit speed curve with respect to $K_\Om$  we have 
		\begin{equation}\label{delta}
		\begin{split}
		\delta_\Om(p,z)\leq& \int_0^a |\gamma_1'(t)| dt\\
		\lesssim& \int_0^a \left(g\left(\frac{1}{\delta(\gamma_1 (t)}\right)\right)^{-1} K_\Om(\gamma_1 (t),\gamma_1' (t))	dt\\
		\lesssim& \int_0^a \left(g\left(e^{2k_\Om(z_0,\gamma_1(t))-2c_0}\right)\right)^{-1}dt\\
			\end{split}
			\end{equation}
			 We now compare $a$ and $4k_{\Om}(z_0,\gamma_1(t))$. In the case  $a>4k_{\Om}(z_0,\gamma_1(t))$, we split the integral into two parts and use the inequalities \eqref{triangle} combining with the increasing of $g$. It gives us  
		
			\begin{equation}	\begin{split}
		\delta_\Om(p,z)\lesssim& \int_0^{4k_{\Om}(z_0,\gamma)}\left(g\left(e^{2k_\Om(z_0,\gamma(t))-2c_0}\right)\right)^{-1}dt+\int_{4k_{\Om}(z_0,\gamma)}^a\left(g\left(e^{2k_\Om(z_0,\gamma(t))-2c_0}\right)\right)^{-1}dt\\
		\lesssim& \int_0^{4k_{\Om}(z_0,\gamma)}\left(g\left(e^{2k_\Om(z_0,\gamma)-2c_0}\right)\right)^{-1}dt+\int_{4k_{\Om}(z_0,\gamma)}^\infty\left(g\left(e^{2t-2k_\Om(z_0,\gamma)-2c_0}\right)\right)^{-1}dt\\
		\lesssim& 	\frac{4k_{\Om}(z_0,\gamma)}{g\left(e^{2k_\Om(z_0,\gamma)-2c_0}\right)}+\int_{e^{2k_{\Om}(z_0,\gamma)-2c_0}}^\infty\frac{d\alpha}{\alpha g(\alpha)}\\
		\lesssim& 	\left(\frac{\ln s}{g(s)}+\int_{s}^\infty\frac{d\alpha}{\alpha g(\alpha)}\right)\Big|_{s=e^{2k_{\Om}(z_0,\gamma)-2c_0}}.
		\end{split}
		\end{equation}
We notice that
$$\int_{s}^\infty\frac{d\alpha}{\alpha g(\alpha)}=\int_s^\infty\frac{1}{\alpha}\left(\int_\alpha^\infty\frac{d\beta }{\beta f(\beta)}\right)d\alpha=\int_s^\infty\frac{1}{\alpha f(\alpha)}\left(\int_s^\alpha\frac{d\beta}{\beta}\right)d\alpha=\int^\infty_s\frac{\ln \alpha-\ln s}{\alpha f(\alpha)}d\alpha,$$
and hence,
$$\frac{\ln s}{g(s)}+\int_{s}^\infty\frac{d\alpha}{\alpha g(\alpha)}=\int_s^\infty\frac{\ln \alpha}{\alpha f(\alpha)}d\alpha.$$
Therefore, in this case we obtain 
$$\delta_\Om(p,z)\lesssim \int_{e^{2k_\Om(z_0,\gamma)-2c_0}}^\infty\frac{\ln \alpha}{\alpha f(\alpha)}d\alpha.$$
In the case $a<4k_{\Om}(z_0,\gamma)$, we make the same estimate but without decomposing the integral. By a symmetric argument with $w$ instead of $z$, we also have 	$$\delta_\Om(p,w)\lesssim \int_{e^{2k_\Om(z_0,\gamma)-2c_0}}^\infty\frac{\ln \alpha}{\alpha f(\alpha)}d\alpha.$$
 The conclusion of this lemma follows by the triangle inequality.
	\end{proof}

	\begin{corollary}\label{C2} Let $\Omega $ be a bouned, pseudoconvex domain in ${\mathbb C}^n$ with $C^2$-smooth boundary satisfying the $f$-property with  $\displaystyle\int_1^\infty\frac{\ln \alpha}{\alpha f(\alpha )}d\alpha<\infty$. Furthermore, assume that $\Om$ is k-complete. Let $\{w_n\},\{z_n\}\subset\Omega$ be two sequence such that $w_n\to x\in \di\Om$ and $z_n\to y\in \bar\Omega\setminus \{x\}$. Then $k_\Om(w_n,z_n)\to \infty$.
\end{corollary}

\begin{proof} Fix a point $z_0\in \Omega$ and let $\gamma_n:[a_n,b_n]\to\Om$ is a minimizing geodesic connecting $z_n=\gamma(a_n)$ and $w_n=\gamma(b_n)$. Since $x\ne y$, it follows $\delta(z_n,w_n)\gtrsim 1$.  By Lemma \ref{Lem1}, 
	it follows
$$	1\lesssim \int^\infty_{e^{k_\Om(z_0,\gamma_n)-2c_0}}\frac{\ln \alpha}{
\alpha f(\alpha)}d\alpha.$$
	This inequality implies that $k_\Om(z_0,\gamma_n)\lesssim 1$ because the function $\displaystyle \int_{s}^\infty\frac{\ln \alpha}{\alpha f(\alpha)}d\alpha$ is decreasing. It means that there is a point $p_n\in \gamma_n $ such that $k_\Om(z_0,p_n)=k_\Om(z_0,\gamma_n)\lesssim 1$. Moreover, 
\begin{equation*}
\begin{split}
k_\Om(z_0,w_n)&\leq k_\Om(z_0,p_n)+k_\Om(p_n,w_n)\\
                             & \leq k_\Om(z_0,p_n)+k_\Om(w_n,z_n)\\
                              &\lesssim k_\Om(w_n,z_n)+1.
\end{split}
\end{equation*}
Since $\Om$ is k-complete, this implies $k_\Om(z_0,w_n)\to \infty$ as $w_n\to x\in \di\Om$. This proves Corollary~\ref{C2}.
\end{proof}
\section{Proof of Theorem~\ref{main}}
In order to give the proof of Theorem~\ref{main}, we  recall the definition of small, big horospheres and $F$-convex in \cite{A1, A3}.
\begin{definition}(see \cite[p.228]{A1}) Let $\Omega$ be a domain in $\mathbb C^n$ . Fix $z_0\in\Omega,\ x \in \p \Omega$ and $R > 0.$ Then
the small horosphere $E_{z_0}(x, R)$ and the big horosphere $F_{z_0}(x, R)$ of center $x,$ pole $z_ 0$
and radius $R$ are defined by
$$E_{z_0} (x, R) = \{z\in \Omega\colon \limsup_{\Omega \ni w\to x} [ k_\Om(z, w) -k_\Om(z_0, w)] < \frac{1}{2}\log  R \},$$
$$F_{z_0} (x, R) = \{z\in \Omega\colon \liminf_{\Omega \ni w\to x} [ k_\Om(z, w) - k_\Om(z_0, w)] < \frac{1}{2}\log  R \}.$$
\end{definition}
\begin{definition}(see \cite[p.185]{A3})
 A domain $\Omega\subset \mathbb C^n$ is called $F$-convex if for every
$x \in \partial \Omega$
$$  \overline{F_{z_0}(x, R)}\cap \di\Omega \subseteq \{x\}$$
holds for every $R > 0$ and for every $z_0\in \Omega$. 
\end{definition}
\begin{remark} The bidisk $\Delta^2$ in $\C^2$ is not $F$-convex. Indeed, since $d_{\Delta^2}((1/2,1-1/k),(0,1-1/k))-d_{\Delta^2}((0,0),(0,1-1/k))=d_\Delta(1/2,0)-d_\Delta(0,1-1/k) \to -\infty$ as $\mathbb N^*\ni k\to \infty$,  $ (1/2,1)\in\overline{F^{\Delta^2}_{(0,0)}((0,1), R)}\cap \partial(\Delta^2)$ for any $R>0$.
\end{remark}

\begin{remark} If $\Om$ is a strongly pseudoconvex domain in $\C^n$, or pseudoconvex domains of finite type in $\C^2$, or a domains of strict finite type in $\C^n$ then $\Om$ is $F$-convex (see \cite{A1,A3, Zha}).
\end{remark}     
Now, we prove that $F$-convexity holds on a larger class of pseudoconvex domains.
\begin{proposition} \label{T9} Let $\Omega$ be a domain satisfying the hypothesis in Theorem~\ref{main}. Then $\Omega$  is $F$-convex.
\end{proposition}
\begin{proof} Let $R>0$ and $z_0\in\Om$. Assume by contradiction that there exists $y\in  \overline{F_{z_0}(x,R)}\cap \di\Omega$ with  $y \ne x$. Then we can find a sequence $\{z_n\}\subset \Omega$ with $z_n \to y\in\di\Om$ and a sequence $\{w_n\}\subset \Omega$ with $w_n \to x\in\di\Om$ such that 
\begin{equation}\label{eqt1}
k_\Om(z_n,w_n)-k_\Om(z_0,w_n)\leq\frac{1}{2}\log R .
\end{equation}
Moreover, for each $n\in \mathbb N^*$ there exists a minimizing geodesic $\gamma_n$ connecting $z_n$ to $w_n$. Let $p_n$ be a point on $\gamma_n$ of minimal distance $k_\Om(z_0,\gamma_n)=k_\Om(z_0,p_n)$ to $z_0$. We consider two following cases of the sequence $\{p_n\}$.
\smallskip

\noindent
{\bf Case 1.} If there exists a subsequence $\{p_{n_k}\}$ of the sequence $\{p_n\}$ such that $p_{n_k}\to p_0\in \Omega$ as $k\to \infty$. 
\begin{equation}\label{eq2}
\begin{split}
k_\Om(w_{n_k},z_{n_k})&\gtrsim k_\Om(w_{n_k},p_{n_k})+k_\Om(p_{n_k},z_{n_k})\\
&\gtrsim k_\Om(w_{n_k},z_0)-k_\Om(z_0,p_{n_k})+k_\Om(p_{n_k},z_{n_k}).
\end{split}
\end{equation}
From (\ref{eqt1}) and (\ref{eq2}), we obtain
$$k_\Om(p_{n_k},z_{n_k})\lesssim k_\Om(w_{n_k},z_{n_k})-k_\Om(w_{n_k},z_0)+k_\Om(z_0,p_{n_k}) \lesssim \frac{1}{2}\log R+k_\Om(z_0,p_{n_k})\lesssim 1. $$
This is a contradiction since $\Omega$ is $k$-complete.
\smallskip

\noindent
{\bf Case 2.} Otherwise, $p_n\to \partial \Omega$ as $n\to\infty$.  By Lemma \ref{Lem1}, there are constants $c$ and $c_0$ only depending on $z_0$ such that
\begin{equation}\label{new1}
\delta_\Om(w_n, z_n)\leq c\int_{e^{2k_{\Om}(z_0,\gamma_n)-2c_0}}^{+\infty}\frac{\ln\alpha }{\alpha f(\alpha)}d\alpha.
\end{equation}
On the other hand, $\delta_\Om(w_n, z_n)\gtrsim 1$ since $x\not=y$. Thus, the inequality \eqref{new1} implies that 
\begin{equation}\label{eqt3}
k_\Om(z_0,\gamma_n)=k_\Om(z_0,p_n)\lesssim 1.
\end{equation}
Therefore,  
\begin{equation}\label{eqt4}
\begin{split}
k_\Om(z_n,w_n) & \gtrsim k_\Om(z_n,q_n)+k_\Om(q_n,w_n)\\
                       &\gtrsim k_\Om(z_0, z_n)+k_\Om(z_0,w_n)-2k_\Om(z_0,q_n).
\end{split}
\end{equation}
Combining with (\ref{eqt1}) and (\ref{eqt3}), we get
$$k_\Om(z_0,z_n)\lesssim k_\Om(z_n,w_n)-k_\Om(z_0,w_n)+2k_{\Om}(z_0,q_n)\lesssim \log R+1.$$
This is a contradiction since $z_n\to y\in\di\Om$ and hence the proof completes.
\end{proof}
 
The following theorem is a generalization of Theorem 3.1 in \cite{A3}.
\begin{proposition}\label{T10} Let $\Omega$ be a domain satisfying the hypothesis in Theorem~\ref{main} and fix $z_0\in \Omega$. Let $\phi\in H(\Om,\Om)$ such that $\{\phi^k\}$ is compactly divergent. Then there is a point  $x\in\partial\Omega$ such that for all $R>0$ and for all $m\in \mathbb N$
$$ \phi^m(E_{z_0}(x,R))\subset F_{z_0}(x,R). $$
\end{proposition}
\begin{proof}
Since $\{\phi^k\}$ is compactly divergent and $\Om$ is $k$-complete,
$$ \lim_{k\to + \infty}k_\Om(z_0,\phi^k(z_0))=\infty. $$
For every $\nu\in \mathbb N$, let $k_\nu$ be the largest integer $k$ satisfying $k_\Om(z_0,\phi^k(z_0))\leq \nu$; then 
\begin{equation}\label{e1}
k_\Om(z_0,\phi^{k_\nu}(z_0))\leq \nu<k_\Om(z_0,\phi^{k_\nu+m}(z_0))\; \forall \nu\in\mathbb N,\; \forall m>0.
\end{equation}
Again, since $\{\phi^k\}$ is compactly divergent, up to a subsequence, we can assume that 
$$\phi^{k_\nu}(z_0)\to x\in \di\Omega.$$
 Fix an integer $m\in\mathbb N$, 
then without loss of generality we may assume that $
\phi^{k_\nu}(\phi^m(z_0))\to y\in \di\Omega$. Using the fact that 
$$k_\Om(\phi^{k_\nu}(\phi^m(z_0)),\phi^{k_\nu}(z_0))\leq k_\Om(\phi^m(z_0), z_0)\qquad\T{(by \eqref{contracted})}$$ and results in Corollary \ref{C2}, it must hold that $x=y$. \\

Set $w_\nu=\phi^{k_\nu}(z_0)$. Then $w_\nu\to x$ and $\phi^m(w_\nu)=\phi^{k_\nu}(\phi^m(z_0))\to x$. From \eqref{e1}, we also have
\begin{equation}\label{e2}
\limsup_{\nu\to +\infty}[k_\Om(z_0,w_\nu)-k_\Om(z_0,\phi^p(w_\nu))]\leq 0,
\end{equation}
Now, fix $m>0,\; R>0$ and take $z\in E_{z_0}(x,R)$. We obtain 
\begin{equation}\label{e3}
\begin{split}
&\liminf_{\Omega \ni w\to x}[k_\Om(\phi^m(z),w)-k_\Om(z_0,w)]\\
&\leq \liminf_{\nu\to +\infty}[k_\Om(\phi^m(z),\phi^m(w_\nu))-k_\Om(z_0,\phi^m(w_\nu))]\\
&\leq \liminf_{\nu\to +\infty}[k_\Om(z,w_\nu)-k_\Om(z_0,\phi^m(w_\nu))]\\
&\leq \liminf_{\nu\to +\infty}[k_\Om(z,w_\nu)-k_\Om(z_0,\phi_\nu)]\\
&\;\; +\limsup_{\nu\to +\infty}[k_\Om(z_0,w_\nu)-k_\Om(z_0,\phi^m(w_\nu))]\\
&\leq \liminf_{\nu\to +\infty}[k_\Om(z,w_\nu)-k_\Om(z_0,w_\nu)]\\
&\leq \limsup_{\Omega \ni w\to x}[k_\Om(z,w)-k_\Om(z_0,w)]\\
&<\frac{1}{2}\log R,
\end{split}
\end{equation}
that is $\phi^m(z)\in F_{z_0}(x,R)$. Here, the first inequality follows by $\phi^p(w_\nu)\to x$, the second follows  by \eqref{contracted}, the fourth follows by \eqref{e2}, and the last one follows by $z\in E_{z_0}(x,R)$.
\end{proof}
\begin{lemma}\label{Lem2} Let $\Omega$ be a $F$-convex domain in $\C^n$. Then for 
any $x,y\in \partial\Omega$ with $x\ne y$ and for any $R>0$, we have
$ \lim\limits_{a\to y}E_{a}(x,R)=\Omega$, i.e., for each $z\in\Omega$, 
there exists a number $\epsilon>0$ such that $z\in E_{a}(x, R)$ for all $a \in\Omega$ 
with $|a-y|<\epsilon$.
\end{lemma}
\begin{proof}
Suppose that there exists $z\in \Omega$ such that there exists a sequence $\{a_n\}\subset \Omega$ with $a_n\to y$ and $z\not \in E_{a_n}(x,R)$.  Then we have
$$\limsup_{w\to x}[k_\Om(z,w)-k_\Om(a_n,w)]\geq \frac{1}{2}\log R.$$
This implies that
$$\liminf_{w\to x}[k_\Om(a_n,w)-k_\Om(z,w)]\leq \frac{1}{2}\log \frac{1}{R}.$$
Thus, $a_n\in \overline{F_z(x,1/R)}$, for all $n=1,2,\cdots$. Therefore, $y\in\overline{F_z(x,1/R)}\cap \partial\Omega=\{x\}$, which is absurd. This ends the proof.
\end{proof}
Now we are ready to prove our main result.
\begin{proof}[{\bf Proof of Theorem \ref{main}}]First we fix a point $z_0\in\Om$, by Proposition \ref{T10} there is a point  $x\in\partial\Omega$ such that for all $R>0$ and for all $m\in \mathbb N$
$$ \phi^m(E_{z_0}(x,R))\subset F_{z_0}(x,R). $$
 We need to show that for any $z\in\Om$
$$\phi^{m}(z)\to x\quad\T{ as }\quad m\to +\infty.$$ 
Indeed, let $\psi(z)$ be a limit point of $\{\phi^{m}(z)\}$. 
Since $\{\phi^{m}\}$ is compactly divergent, $\psi(z)\in \di\Omega$. By Lemma \ref{Lem2}, for any $R>0$ there is $a\in \Omega$ such that $z\in E_{a}(x,R)$. By Proposition \ref{T10}, $\phi^m(z)\in F_{a}(x,R)$ for every $m\in\mathbb N^*$. Therefore, $$\psi(z) \in \di\Omega\cap \overline{F_{a}(x,R)}=\{x\}$$ 
by  Proposition \ref{T9}; thus the proof is complete.
\end{proof}

\section*{Acknowledgments} The research of the second author was supported in part by a grant of Vietnam National University at Hanoi, Vietnam. This work was completed when the second author was visiting the Vietnam Institute for Advanced Study in Mathematics (VIASM). He would like to thank the VIASM for financial support and hospitality.

\end{document}